\documentclass[a4paper,10pt]{article}
\usepackage[utf8]{inputenc}
\usepackage[T1]{fontenc}
\usepackage[utf8]{inputenc}
\usepackage{amsthm}
\usepackage{amsmath}
\usepackage{hyperref}
\usepackage{xspace}
\usepackage{graphicx}
\usepackage[rmargin=3cm,lmargin=3cm,tmargin=2cm,bmargin=3cm]{geometry}
\usepackage{amssymb}
\usepackage{textcomp}
\usepackage{wasysym}
\usepackage{xcolor}

\newtheorem{definition}{Definition}[section]
\newtheorem{proposition}[definition]{Proposition}

\newtheorem{theorem}[definition]{Theorem}
\newtheorem{corollary}[definition]{Corollary}
\newtheorem{example}{Example}
\newtheorem{lemma}[definition]{Lemma}
\title{$p$-Laplacian first eigenvalues controls on  Finsler manifolds}
\author{Cyrille Combete, Serge Degla, Leonard Todjihounde}
\date{}
\begin{document}

\maketitle

\begin{abstract}
Given a Finsler manifold $(M,F)$, it is proved that the first eigenvalue of the Finslerian $p$-Laplacian is bounded above by a constant depending 
on $\ p$, the dimension of $M$, the Busemann-Hausdorff volume and  the reversibility constant of $(M,F)$.

For a Randers manifold $(M,F:=\sqrt{g}+\beta)$, where $g$ is a Riemannian metric on $M$ and $\beta$ an appropriate $1$-form on $M$, 
it is shown that the first eigenvalue $\lambda_{1,p}(M,F)$ of the Finslerian $p$-Laplacian defined by the Finsler metric $F$ 
is controled by the first eigenvalue 
$\lambda_{1,p}(M,g)$  of the Riemannian  $p$-Laplacian defined on $(M,g)$.

Finally, the Cheeger's inequality  for Finsler Laplacian is extended for $p$-Laplacian, with $p > 1$.
\end{abstract}

\section{Introduction} \label{intro}


The study of the $p$-Laplace operator $-$ and in particular of its first eigenvalue $-$ is a classical and important problem in Riemannian 
geometry. In \cite{Mat05,Mat12}, the author studies the first eigenvalue of the $p$-Laplacian $\Delta_p$ on a compact Riemannian manifold $M$ as a functional 
on the space of Riemannian metrics on $M$. He proved that on any compact manifold of dimension $n\geq 3$, 
there is a Riemannian metric of volume one such that the first 
eigenvalue of the $p$-Laplacian can be taken arbitrary large
and that the eigenvalue functional restried to the conformal class is bounded above for $1<p\leq n$.

 In Finsler geometry, there is no canonical way to introduce the Laplacian. 
Hence, several authors proposed different extensions of the standard Riemannian Laplacian to the Finsler setting like
Antonelli and Zastawniak  \cite{AZ}, Bao and Lackey  \cite{BL}, Barthelm\'e \cite{Bar}, 
Centor\'e \cite{Cen} and Shen \cite{Sh98}. In the last decade, the non-linear Shen's Finsler-Laplacian received a particular attention and
 Q. He and S-T Yin use it to introduce the $p$-Laplacian on Finsler manifolds \cite{HY14,HY16}. 
They established some inequalities related to the first eigenvalue and obtained 
a regularity theorem of its associated functions.

Eigenfunctions of the $p$-Laplacian have weaker regularities in the Finslerian setting than the Riemannian
one, due to the non-linearity of the Finslerian $p$-Laplacian. 

In \cite{MT1}, the author shows that a canonical smooth Riemannian metric can be associated to any Finsler metric $F$. 
This Riemannian metric is called Binet-Legendre metric and is bi-lipschitz equivalent to $F$ with lipschitz constant depending only on the dimension of the manifold and on the 
reversibility constant of $F$ (see Section 2.3). It allows  
us to control the first eigenvalue of the Finsler $p$-Laplacian and to prove our main result:

\begin{theorem}\label{A}
Let $(M,F)$ be a compact Finsler $n$-dimensional manifold. Then,  for any $p\in (1,n]$, there exists a constant $C:=C(n,p,\kappa_F,[F])$ depending only on the dimension $n$, $p$, 
 the reversibility constant $\kappa_F$ and the conformal class 
 $[F]$ of $F$ such that,
 \[
  \lambda_{1,p}(M,F)Vol(M,F)^{\frac{p}{n}}\leq C(n,p,\kappa_F,[F]). 
 \]
\end{theorem}

Randers metrics are an important class of Finsler metrics. They are Finsler metrics of the form $F:=\sqrt{g}+\beta$ 
where $g$ is a Riemannian metric and $\beta$ a 
$1$-form which norm with respect to the metric $g$ is smaller than one. 
It is interesting to know the relations between geometric quantities related to 
$F$ and $g$ respectively. We prove the following

\begin{theorem}\label{B}
If $(M,F:=\sqrt{g}+\beta)$ is a Randers manifold endowed with the Holmes-Thompson volume form $d\mu_{HT}$ then, for any $p>1$, we have
 \[
  \frac{1}{\kappa_F^p}\lambda_{1,p}(M,g)\leq \lambda_{1,p}(M,F) \leq \kappa_F^p \lambda_{1,p}(M,g),
 \]
where $\lambda_{1,p}(M,g)$ is the first eigenvalue of the $p$-Laplacian on the Riemannian manifold $(M,g)$ and $\kappa_F$, the reversibility 
constant of $(M,F)$.
\end{theorem}

In \cite{Ch}, Cheeger introduced for a closed Riemannian manifold $(M,g)$
an geometric invariant $\mathbf{h}(M)$ called Cheeger invariant, and he proved that $4 \lambda_{1,2}(M)\geq \mathbf{h}^2(M)$. The authors in 
\cite{YZ13} generalize this inequality for the Finslerian Laplacian. In this paper we extend their result to the Finslerian $p$-Laplacian 
for $p > 1$.

The content of the paper is organized as follows. In section \ref{section2} , we recall some fundamental notions which are necessary and 
important for this article. Section \ref{section3} and \ref{section4} are devoted to the proofs of Theorem \ref{A} and Theorem \ref{B} respectively. We prove the 
Cheeger's type inequaliy in the last section.

\section{Preliminaries}\label{section2}
Let $M$ be a connected, $n$-dimensional smooth manifold without boundary. Given a local coordinate system $(x^i)_{i=1}^n$ on an open set $U$ of 
$M$, we will use 
the coordinate $(x^i,v^i)_{i=1}^n$ of $TU$ such that for all $v\in T_xM$, $ x\ \in U$,
\[
 v:= v^i\left.\frac{\partial}{\partial x^i}\right|_x.
\]

\subsection{Finsler geometry}
\begin{definition}\label{Finsler}
 A Finsler metric on $M$ is a nonnegative function $F: TM\rightarrow [0,\infty)$ satisfying:
  \begin{enumerate}
   \item (Regularity) $F$ is $C^\infty$ on $TM\backslash O$, where $O$ stands for the zero section,
   \item (Positive 1-homogeneity) It holds $F(cv)=cF(v)$ for all $v \in TM$ and $c\geq 0$,
   \item (Strong convexity) The $n\times n$ matrix 
   \begin{equation}\label{eq:2}
    (g_{ij}(v))_{1\leq i,j\leq n}:=(\frac{1}{2}\frac{\partial^2(F^2)}{\partial v^i \partial v^j}(v))_{1\leq i,j\leq n}
   \end{equation}
is positive-definite for all $v \in T_xM\backslash \{0\}$.
  \end{enumerate}
  \end{definition}

 Remark that   for each $v \in T_xM\backslash\{0\}$, the positive-definite matrix $(g_{ij}(v))_{1\leq i,j\leq n}$ in the Definition \ref{Finsler} defines
 the Riemannian structure 
 $g_v$ of $T_xM$ via 
 \[
  g_v\left(\sum_{i=1}^n a_i \frac{\partial}{\partial x^i},\sum_{j=1}^n b_j \frac{\partial}{\partial x^j}\right):= \sum_{i,j=1}^n g_{ij}(v) a_ib_j.
 \]

 The \emph{reversibility constant} of $(M,F)$ is defined by 
 \[
  \kappa_F:=\sup_{x\in M}\sup_{v\in T_xM\backslash\{0\}}\frac{F(v)}{F(-v)} \in [1,\infty].
 \]
$F$ is said to be reversible if $\kappa_F=1$, that is $F(v)=F(-v),\ \forall\ x\in T_xM $.

\vspace{0.2cm}
 The dual metric $F^*:T^*M\rightarrow [0,\infty)$ of $F$ on $M$ is defined for any $\alpha \in T^*M$ by
 \[
  F^*(\alpha):= \sup_{v\in T_xM,F(v)\leq 1 }\alpha(v) = \sup_{v \in T_xM,F(v)=1}\alpha(v).
 \]

 One also define the $2$-uniform concavity constant as
 \[
\sigma_F:=\sup_{x\in M}\sup_{v,w \in T_xM\backslash\{0\}}\frac{g_v(w,w)}{F(w)^2}=\sup_{x\in M}\sup_{\alpha,\beta \in T^*_xM\backslash\{0\}}\frac{F^*(\beta)}{g^*_\alpha(\beta,\beta)} 
\in [1,\infty].
\]
 $F$ is Riemannian if and only if $\sigma_F=1$ (see \cite{Oht16}).

 \vspace{0.5cm}
Given a vector field $X:= X^i \frac{\partial}{\partial x^i}$, the covariant derivate of $X$ by $v\in T_xM$ with the reference $w\in T_xM\backslash\{0\}$ 
is defined by
\[
 D_v^w X(x):=\left\{ v^j \frac{\partial X^i}{\partial x^j}(x) + \Gamma^i_{jk}(w)v^jX^k(x) \right\} \frac{\partial}{\partial x^i},
\]
where $\Gamma_{jk}^i(w)$ are the coefficients of the Chern connection.

The flag curvature of the plane spanned by two linearly independent vector $V,W \in T_xM\backslash\{0\}$ is given by 
\[
 K(V,W):=\frac{g_V(R^V(V,W)W,V)}{g_V(V,V)g_V(W,W)-g_V(V,W)^2},
\]
where $R^V$ is the Chern curvature:
\[
 R^V(X,Y)Z:= D_X^VD_Y^VZ+D_Y^VD_X^VZ-D_{[X,Y]}^VZ.
\]
The Ricci curvature of $(M,F)$ is defined by 
\[
 Ric(V):= \sum_{i=1}^{n-1} K(V,e_i),
\]
where $\{e_1, e_2, \dots, e_n=\frac{V}{F(V)}  \}$ is an orthonormal basis of $T_xM$ with respect to $g_V$.

\subsection{Finsler p-laplacian}

Denote by $J^*:T^*M\rightarrow TM$ the Legendre transform which assigns to each $\alpha \in T^*_xM$ the unique maximizer of the function 
$v\mapsto \alpha(v)-\frac{1}{2}F^2(x,v)$ on $T_xM$. The quantity $J^*(x,\alpha)$ is characterized as the unique vector $v\in T_xM$ with $F(x,v)=F^*(x,\alpha)$ and 
$\alpha(v)=F^*(x,\alpha)F(x,v)$.

For a differentiable function $f:M\rightarrow \mathbb{R}$, the gradient vector of $f$ at $x$ is define as the Legendre transform of the derivate of $f$:
$\nabla f(x):= J^*(x,Df(x))$. In coordinates, we have
\[
 \nabla f(x)=  \left\{
	\begin{array}{ll}
	 g^{ij}(x,Df(x))\frac{\partial f}{\partial x^j}\frac{\partial}{\partial x^i},\ \text{if}\  df(x)\neq 0  \\
	 0, \hspace{3cm}            \ \text{if}\    df(x)=0
	 
	\end{array}
  \right.
\]
where $g^{ij}(x,\alpha):= \frac{1}{2}\frac{\partial^2F^*(x,\alpha)^2}{\partial\alpha^i\partial\alpha^j}$. Remark that 
$(g^{ij}(x,\alpha))_{ij}$ is the inverse matrix of $(g_{ij}(x,J^*(x,\alpha)))_{ij}$.


 
 \vspace{0.5cm}
  We fix an abitrary positive $C^\infty$-measure $\mathfrak{m}$ on $M$ as our base measure. In a local coordinates system, 
  the measure element is given by 
  $d\mathfrak{m}:=e^{\Phi} dx^1\dots dx^n$. Usually, the Busemann-Hausdorff volume form $d\mathfrak{m}_{BH}$ and the Holmes-Thompson volume form 
  $d\mathfrak{m}_{HT}$ are used. They are defined by 
\[
 d\mathfrak{m}_{BH}:= \frac{\omega_n}{Vol(B_xM)}dx^1 \wedge \dots \wedge dx^n,
\]
and
\[
 d\mathfrak{m}_{HT}:= \left( \frac{1}{\omega_n} \int_{B_xM} det g_{ij}(x,v) dv^1\wedge\dots\wedge dv^n \right) dx^1\wedge\dots\wedge dx^n,
\]
 where $B_xM:=\{ v\in T_xM: F(x,v)<1\}$ and $\omega_n$ denotes the volume of $n$-dimensional euclidean ball.

  The divergence of a differentiable vector field $V$ on $M$ with respect to $\mathfrak{m}$ is defined by
  \[
   div_{\mathfrak{m}}V:=\sum_{i=1}^n \left(\frac{\partial V^i}{\partial x^i}+V^i\frac{\partial \Phi}{\partial x^i}\right).
  \]

  Denote by $W^{1,p}(M)$ the completion of $C^\infty(M)$. For a function $f\in W^{1,p}(M)$, its Finsler p-Laplacian is defined as
  \[
   \Delta_p(f):=div_\mathfrak{m}(F(\nabla f)^{p-2}\nabla f):= div_\mathfrak{m}(|\nabla f|^{p-2}\nabla f),
  \]
where the equality is in the distibutional sense.

For $p=2$, we obtain the non-linear Shen's Finsler Laplacian:
\[
 \Delta_2 (f ):=\Delta (f) = div_{\mathfrak{m}}(\nabla f).
\]

This operator is naturally associated to the canonical energy functional $E$ defined on $W^{1,p}(M)\backslash\{ 0\}$ by 
\[
 E(f):= \frac{\int_M |\nabla f|^p\ d\mathfrak{m} }{\int_M |f|^p\ d\mathfrak{m}}.
\]

The first (closed) eigenvalue of the Finsler $p$-Laplacian is defined by
\[
 \lambda_{1,p}(M):= \inf_{f \in \mathcal{H}^p_0} E(f),
\]
where $\mathcal{H}^p_0:=\{ f \in W^{1,p}(M)\backslash \{0\}:\ \int_M |f|^{p-2}f\ d\mathfrak{m}=0 \}$. An eigenfunction related to the first 
eigenvalue is a function $f \in W^{1,p}(M)$ which satisfies $\Delta_p f + \lambda_{1,p}(M) |f|^{p-2}f=0$. We have the following caracterisation: for all
$\varphi \in W^{1,p}(M),$
\[
 \int_M |\nabla f|^{p-2} d\varphi(\nabla f)\ d\mathfrak{m} = \lambda_{1,p}(M) \int_M |f|^{p-2}f\varphi\  d\mathfrak{m}.
\]

\vspace{0.5cm}

Now, we will  recall the construction of a canonical Riemannian metric associated to the Finsler manifold $(M,F)$. See \cite{MT1,MT2} for more details.


\subsection{Binet-Legendre metric}

In this part, $d\mathfrak{m}_F$ will always denote the Busemann-Hausdorff measure induced by the metric $F$ on $M$. 

Let define a scalar product on the cotangent spaces $T_x^*M$, $(x\in M)$ by
\[
 g^*_F(\alpha,\beta):=\frac{n+2}{\lambda(B_xM)}\int_{B_xM} \alpha(v).\beta(v)\ d\lambda(v),
\]
where $\lambda$ is a Lebesgue measure on $T_xM$.

The Binet-Legendre metric $g_F$ associated to the Finsler metric $F$ is the Riemannian metric dual to the scalar product $g_F^*$ .

\begin{theorem}\cite{MT2}\label{binet}
 Let $(M,F)$ be a Finsler $n$-manifold with finite reversibility constant $\kappa_F$ and $g_F$ its associated Binet-Legendre metric. Then
 \begin{itemize}
  \item[(i)] The metric $g_F$ is as smooth as $F$;
  \item[(ii)] We have 
  \[
   (\kappa_F \sqrt{2n})^{-n-1} \sqrt{g_F}\leq F \leq (\kappa_F\sqrt{2n})^{n+1}\sqrt{g_F};
  \]
  \item[(iii)] If $dV_{g_F}$ denotes the Riemannian volume density of $g_F$, there is a constant $k$ such that
  \[
   \omega_n k^{-n} dV_{g_F}\leq d\mathfrak{m}_F\leq \omega_n k^n dV_{g_F},
  \]
  where $\omega_n$ denote the volume of the standard $n$-dimensional Euclidean ball.
In particular, $dV_{g_F} \leq d\mathfrak{m}_F$.
 \end{itemize}

\end{theorem}

\begin{proposition}\label{pp:binet}
 Let $(M,F)$ be a Finsler $n$-manifold with finite reversibility constant $\kappa_F$ and $g_F$ its associated Binet-Legendre metric. Then
 \[
  \frac{1}{(\kappa_F\sqrt{2n})^{p(n+1)}\omega_nk^n} \leq \frac{\lambda_{1,p}(M,F)}{\lambda_{1,p}(M,g_F)}\leq (\kappa_F\sqrt{2n})^{p(n+1)}\omega_nk^n.
 \]

\end{proposition}

\begin{proof}

 Let $f\in C^\infty(M)$ such that $\int_M |f|^pd\mathfrak{m}_F\neq 0$.
 
 From theorem \ref{binet}, we have
 \begin{eqnarray*}
  \frac{1}{(\kappa_F\sqrt{2n})^{p(n+1)}\omega_nk^n}\frac{\int_M \|\nabla f\|^p\ dV_{g_F}}{\int_M |f|^p\ dV_{g_F}} &\leq& 
  \frac{\int_M F(\nabla f)^p\ d\mathfrak{m}_F}{\int_M |f|^p\ d\mathfrak{m}_F}\\
  &\leq& (\kappa_F\sqrt{2n})^{p(n+1)}\omega_nk^n\frac{\int_M \|\nabla f\|^p\ dV_{g_F}}{\int_M |f|^p\ dV_{g_F}},
 \end{eqnarray*}
which yields the result.
 
\end{proof}

\begin{definition}
 Two Finsler metrics $F_0$ and $F$ defined on a smooth manifold are called bi-Lipschitz if there exists a constant $C>1$ such that, for any $(x,v)\in TM$, 
 \begin{equation}\label{eq:main}
  C^{-1}F_0(x,v)\leq F(x,v)\leq CF_0(x,v).
 \end{equation} 
\end{definition}

\begin{example}
   Let $(M,g)$ be a Riemannian manifold and $\beta_1,\beta_2$ such that 
  \[
  0\leq \sup_{x \in M}\|(\beta_1)_x\|_g :=b_1 \leq b_2 := \sup_{x\in M} \|(\beta_2)_x\|_g <1.
  \]
  Then the Randers metrics $F_1:=\sqrt{g}+\beta_1$ and $F_2:=\sqrt{g}+\beta_2$ are bi-Lipschitz: 
  \[
   \frac{1-b_2}{1+b_1}\leq \frac{F_1}{F_2} \leq \frac{1+b_1}{1-b_2}.
  \]
Particulary, a Randers metric $F=\sqrt{g}+\beta$ and the associated Riemannian metric $g$  are bi-Lipschitz.

\end{example}

\begin{lemma}\cite{MT2}\label{lm:binet}
 If $F$ and $F_0$ are Finsler metrics on $M$ satisfying (\ref{eq:main}) for some constant $C$ then the Binet-Legendre metrics $g_F$ and $g_{F_0}$ associated 
 to $F$ and $F_0$ respectively satisfy 
 \[
  C^{-n}\sqrt{g_{F_0}}\leq \sqrt{g_F}\leq C^n\sqrt{g_{F_0}}. 
 \]

 \end{lemma}

\begin{theorem}\label{main} 
 Let $F,\ F_0$ be two $C$-bi-Lipschitz Finsler metrics on a compact $n$-manifold $M$.
 
Then, for any $p>1$, there exists a constant $K(n,p,\kappa,\kappa_0) \geq 1$ depending on $p$, 
the dimension $n$ and the reversibility constants $\kappa$ and $\kappa_0$ 
 of $F$ and $F_0$ respectively such that,  
\[
 C^{-K}\leq \frac{\lambda_{1,p}(M,F)}{\lambda_{1,p}(M,F_0)}\leq C^K.
\]

\end{theorem}

\begin{proof}
  If $g,\ g_0$ are two Riemannian metrics on $M$ such that, for all $v\in T_xM$,
  \[
   \frac{1}{a}\leq \frac{\sqrt{g(v,v)}}{\sqrt{g_0(v,v)}}\leq a,
  \]
for some constant $a>1$, then 
\begin{equation}\label{bli}
\frac{1}{a^{p+2n}}\leq \frac{\lambda_{1,p}(M,g)}{\lambda_{1,p}(M,g_0)}\leq a^{p+2n}.
\end{equation}
 
 Indeed,
 note that for such metrics, their Riemannian volume densities satisfy 
\[
a^{-n}dV_{g_0}\leq dV_g \leq a^n dV_{g_0} . 
\]
Hence, 
for all $f\in C^\infty(M)$ such that $\int_M |f|^p\ dV_g\neq 0$, we obtain
\[
 \frac{1}{a^{p+2n}}\frac{\int_M \|\nabla f\|_{g_0}^p\ dV_{g_0}}{\int_M | f|^p\ dV_{g_0}}\leq
 \frac{\int_M \|\nabla f\|_{g}^p\ dV_{g}}{\int_M | f|^p\ dV_g} \leq 
 a^{p+2n}\frac{\int_M \|\nabla f\|_{g_0}^p\ dV_{g_0}}{\int_M |f|^p\ dV_{g_0}},
\]
which provides (\ref{bli}).

Now, let denote $\kappa,\ \kappa_0$ the reversibility constants of $F$ and $F_0$ respectively. 
 Applying  Proposition 
\ref{pp:binet} to $(M,F)$ and $(M,F_0)$, we obtain
\[
 \frac{1}{(2n\kappa\kappa_0)^{p(n+1)}\omega_n^2k^{2n}}\frac{\lambda_{1,p}(M,g_F)}{\lambda_{1,p}(M,g_{F_0})}\leq 
 \frac{\lambda_{1,p}(M,F)}{\lambda_{1,p}(M,F_0)} \leq (2n\kappa\kappa_0)^{p(n+1)}\omega_n^2k^{2n} \frac{\lambda_{1,p}(M,g_F)}{\lambda_{1,p}(M,g_{F_0})}.
\]
Furthermore, from the claim and Lemma \ref{lm:binet}, we have 
\[
 \frac{1}{C^{n(p+2n)}}\leq \frac{\lambda_{1,p}(M,g_F)}{\lambda_{1,p}(M,g_{F_0})} \leq C^{n(p+2n)}.
\]
Then
\[
 \frac{1}{(2n\kappa\kappa_0)^{p(n+1)}\omega_n^2k^{2n}C^{n(p+2n)}}\leq \frac{\lambda_{1,p}(M,F)}{\lambda_{1,p}(M,F_0)} 
 \leq (2n\kappa\kappa_0)^{p(n+1)}\omega_n^2k^{2n}C^{n(p+2n)}.
\]
Since $(2n\kappa\kappa_0)^{p(n+1)}\omega_n^2k^{2n} >1$, there exist a positive constant $K'(n,p,\kappa,\kappa_0,\omega_n)$ depending on 
$n$, $p$, $\kappa$, $\kappa_0$ and $\omega_n$ such that $(2n\kappa\kappa_0)^{p(n+1)}\omega_n^2k^{2n} \leq C^{K'}$. This completes the proof.
\end{proof}


\section{Boundedness on conformal class}\label{section3}

Let $\mathcal{F}(M)$ be the set of Finsler metrics $F$ on a manifold $M$ with $Vol(M,F)=1$, where 
$Vol(M,F)$ denotes the volume of the Finsler manifold $(M,F)$ with respect to the Busemann-Hausdorff measure induced by $F$. 
The following holds for the first eigenvalues of the $p$-Laplacians, $p > 1$:
\[
 \inf_{F\in \mathcal{F}(M)}\lambda_{1,p}(M,F)=0.\ 
\]
In the Riemannian case the eigenvalues-functional is not generally bounded.
For $p=2$, it is shown that the functional $\lambda_{1,2}$ is bounded when the dimension $n=2$ and is unbounded when $n\geq 3$, 
but $\lambda_{1,2}$ is uniformly bounded when restricted to any conformal class. Matei generalizes these results to any $p>1$ (see 
\cite{Mat05,Mat12}). Using mainly Matei's works and Proposition \ref{pp:binet}, we have the following:

\begin{theorem}\label{conf1}
 Let $(M,F)$ be a compact Finsler $n$-manifold. Then,  for any $p\in (1,n]$, there exists a constant $C:=C(n,p,\kappa_F,[F])$ depending only on the dimension $n$, $p$, 
 the reversibility constant $\kappa_F$ and the conformal class 
 $[F]$ of $F$ such that,
 \[
  \lambda_{1,p}(M,F)Vol(M,F)^{\frac{p}{n}}\leq C(n,p,\kappa_F,[F]). 
 \]

\end{theorem}

Before proving this theorem, let's remark that, in the Mathei's result used (\cite{Mat12}), the dependence on the conformal class of the Riemannian metric come from 
the $n$-conformal volume of the compact Riemannian manifold $(M,g)$ which is defined as 
\[
 V_n^c(M,[g]):=\inf_{\phi \in I_n(M,[g])}\sup_{\gamma \in G_n} Vol(M,(\gamma\circ \phi)^*can),
\]
where $can$ denotes the canonical Riemannian metric on the $n$-dimensional sphere $\mathbb{S}^n$, $G_n:=\{ \gamma \in Diff(\mathbb{S}^n)|\ 
\gamma^*can\in [can]\}$ the group of conformal diffeomorphism of $(\mathbb{S}^n,can)$ and 
$I_n(M,[g]):=\{ \phi: M\rightarrow \mathbb{S}^n|\ \phi^*can\in [g])\}$ the set of conformal immersion from $(M,g)$ to $(\mathbb{S}^n,can)$. 
Using a nice property of the Binet-Legendre metric associated to the Finsler metric $F$, we can obtain a dependence on the conformal class of $F$.

\begin{proof}

 From Proposition \ref{pp:binet}, there is a constant $C_1(n,p,\kappa_F)$ depending only on $n$, $p$ and $\kappa_F$ such that 
 $\lambda_{1,p}(M,F)\leq C_1 \lambda_{1,p}(M,g_F) $. 

 Set $\alpha^{-1} := Vol(M,g_F)^{\frac{2}{n}}$ and $\tilde{g}:=\alpha g_F$. Then, we have 
 \[
 Vol(M,\tilde{g})=\alpha^{\frac{n}{2}}Vol(M,g_F)=1 
 \]
 and 
 \[
 \lambda_{1,p}(M,g_F)=\alpha^{\frac{p}{2}}\lambda_{1,p}(M,\tilde{g}). 
 \]
 Furthermore,
 Matei proved in \cite{Mat12} that there exists a constant 
 $C_2(n,p,[\tilde{g}])$\footnote{In \cite{Mat12}, $C_2=n^{\frac{p}{2}}(n+1)^{|p/2-1|}
 V_n^c(M,[\tilde{g}])$ where $V_n^c(M,[\tilde{g}])$ denote the conformal volume of $(M,\tilde{g})$ } depending on $n$, $p$ 
 and the conformal class of the metric $\tilde{g}$ which satisfy $\lambda_{1,p}(M,\tilde{g})\leq C_2$.
 
 Hence,  by Theorem \ref{binet}, we obtain 
 \[
 \lambda_{1,p}(M,F)Vol(M,F)^{\frac{p}{n}}\leq C_1C_2\left(\frac{Vol(M,F)}{Vol(M,g_F)}\right)^{\frac{p}{n}} \leq 
 C_1C_2 (\omega_n k^n)^{\frac{p}{n}}. 
 \]
 
 It is known that when $F_1$ and $F_2$ are in the same conformal class, 
 then $g_{F_1}$ and $g_{F_2}$ are also in the same conformal class. Hence, the constant $C_1C_2 (\omega_n k^n)^{\frac{p}{n}}$ depend 
 on $n,p,\kappa_F$ and the conformal class $[F]$ 
 of the metric $F$. 

\end{proof}

Particulary, for compact surface, we have the following:

\begin{theorem}
 Let $(\Sigma,F)$ be a compact Finsler surface with genus $\delta$ and reversibility constant $\kappa_F$. Then, for any $1<p\leq 2$, there exists 
 a constant $K(p,\kappa_F)$ depending only on $p$ and $\kappa_F$ such that 
 \[
  \lambda_{1,p}(\Sigma,F)Vol(\Sigma,F)^{\frac{p}{2}} \leq K(p,\kappa_F)\left(\frac{\delta + 3}{2}\right)^{\frac{p}{2}}.
 \]

\end{theorem}

\begin{proof} 

 From the proof of Theorem \ref{conf1}, there exists a constant $A_1(p,\kappa_F)$ depending on $p$ and $\kappa_F$ such that 
 $ \lambda_{1,p}(\Sigma,F)\leq A_1(p,\kappa_F) \alpha^{\frac{p}{2}} \lambda_{1,p}(\Sigma,\tilde{g})$ where $\tilde{g}:=\alpha g_F$ and 
 $\alpha := Vol(\Sigma,g_F)^{-\frac{2}{n}}$. By a result of Matei (see \cite{Mat12}), $\lambda_{1,p}(\Sigma,\tilde{g})\leq C(p)
 \left(\frac{\delta + 3}{2}\right)^{\frac{p}{2}}$ for some constant $C$ depending only on $p$. Then, we have 
\begin{eqnarray}
 \lambda_{1,p}(\Sigma,F)Vol(\Sigma,F)^{\frac{p}{2}} &\leq& A_1C\left(\frac{Vol(\Sigma,F)}{Vol(\Sigma,g_F)}\right)^{\frac{p}{2}}
\left(\frac{\delta + 3}{2}\right)^{\frac{p}{2}}\nonumber\\
 &\leq& A_1(p,\kappa_F)C(p)(\omega_2k^2)^{\frac{p}{2}}\left(\frac{\delta + 3}{2}\right)^{\frac{p}{2}}. 
 \end{eqnarray}
 This completes the proof.

\end{proof}

 \begin{theorem}
  Let $(M,F)$ be a compact Finsler manifold of dimension $n$. Then for any $p>n$, there exists a conformal metric $\tilde{F} \in[F]$  such that the quantity 
  $\lambda_{1,p}(M,\tilde{F})Vol(M,\tilde{F})^{\frac{p}{n}}$ can be taken arbitrarily large.
 \end{theorem}
 
\begin{proof}

 Let $K>0$. From \cite{Mat12}, there exists a metric $\tilde{g}:=\varphi^2g_F \in [g_F]$ satisfying 
 \[
 \lambda_{1,p}(M,\tilde{g})Vol(M,\tilde{g})^{\frac{p}{n}}
 > \frac{K}{C_1} ,
 \]
 for all positive constant $C_1$.  
 Consider the metric $\tilde{F}:=\varphi F \in [F]$. Then the Binet-Legendre metric associated to $\tilde{F}$ is $\tilde{g}$ (see \cite{MT1}). Hence, 
 $\lambda_{1,p}(M,\tilde{F})\geq C(n,p,\kappa_{\tilde{F}}) \lambda_{1,p}(M,\tilde{g})$ (Proposition \ref{pp:binet})
 and $Vol(M,\tilde{F})\geq Vol(M,\tilde{g})$ (Theorem \ref{binet}) . This implies $\lambda_{1,p}(M,\tilde{F})Vol(M,\tilde{F})^{\frac{p}{n}}> K$ taking 
 $C_1=C(n,p,\kappa_{\tilde{F}})$.
 
\end{proof}


\section{Randers spaces}\label{section4}
Consider a Randers metric $F:=\sqrt{g}+\beta$. In local coordinates $(x^i,v^i)$ on $TM$, we write
\[
 g(v,w):=g_{ij}v^iw^j,\ \beta(v)=b_iv^i,\ v=v^i\frac{\partial}{\partial x^i},\ w=w^j\frac{\partial}{\partial x^j}.
\]
Denote $\|\beta\|_x:=\sqrt{g^{ij}(x)b_i(x)b_j(x)}$ and $\mathbf{b}=\sup_{x\in M}{\|\beta\|_x}$ where $(g^{ij})$ stands for the inverse matrix of $(g_{ij})$.

To prove theorem \ref{B}, we need the following lemmas:

\begin{lemma}\cite{Shen}\label{lem:3.1}
 For any smooth function $f$ on $M$, we have 
 \[
  F(\nabla f)= F^*(df)=\frac{\sqrt{(1-\|\beta\|^2)|df|^2+\langle \beta,df \rangle^2}-\langle \beta,df \rangle}{1-\|\beta\|^2},
 \]
where 
\[
 |df|_x:=\sqrt{g^{ij}(x)\frac{\partial f}{\partial x^i}(x)\frac{\partial f}{\partial x^j}(x)},\ and\ 
 \langle \beta,df \rangle_x:=g^{ij}(x)b_i(x)\frac{\partial f}{\partial x^j}(x).
\]

\end{lemma}

\begin{lemma}\cite{YZ13}
 The reversibility constant and the $2$-uniform concavity constant of the Randers space $(M,F:=\sqrt{g}+\beta)$ are given by
 \[
  \sigma_F =\left( \frac{1+\mathbf{b}}{1-\mathbf{b}}\right)^2=\kappa_F^2.
 \]

\end{lemma}

The first eigenvalue of $(M,F)$ and $(M,g)$ can be controlled by the reversibility constant as the next proposition showing. 
Note that a similar result is obtained in \cite{Mun} using Bao-Lackey Laplacian.

\begin{proposition}\label{pp:rand}
 Let $(M,F:=\sqrt{g}+\beta,d\mathfrak{m}_{HT})$ be a Randers space, where $d\mathfrak{m}_{HT}$ is the Holmes-Thompson measure. Then we have
 \[
  \frac{1}{\kappa_F^p}\lambda_{1,p}(M,g)\leq \lambda_{1,p}(M,F) \leq \kappa_F^p \lambda_{1,p}(M,g),
 \]
where $\lambda_{1,p}(M,g)$ is the first eigenvalue of the Riemannian manifold $(M,g)$.
\end{proposition}

\begin{proof}

 Since $d\mathfrak{m}_{HT}$ denotes the Holmes-Thompson measure then it coincides with the Riemannian measure $dV_g$ induced by $g$. Recall that 
 the first eigenvalue on the Riemannian space $(M,g)$ is defined by
 \[
  \lambda_{1,p}(M,g):=\inf_{f\in \mathcal{H}_0^{p}}\frac{\int_M |df|^p\ dV_g}{\int_M |f|^p\ dV_g}.
 \]

 Furthermore, from lemma \ref{lem:3.1}, we have 
 \[
  \frac{1}{\kappa_F} |df| \leq F^*(df) \leq  \kappa_F |df|.
 \]
Indeed, for all  $x\in M$, 

$1-\mathbf{b}\leq 1-\mathbf{b}^2 \leq 1-\|\beta\|_x^2 \leq 1+\mathbf{b}^2 \leq 1+\mathbf{b}$ and  
\begin{eqnarray*}
 \sqrt{(1-\|\beta\|^2)|df|^2+\langle \beta,df \rangle^2}-\langle \beta,df \rangle &\leq& |df| + 2 |\langle \beta,df \rangle|\\
                                                                                  &\leq& (1+2\mathbf{b})|df|.
\end{eqnarray*}
Then
\[
 F^*(df)\leq \frac{1+2\mathbf{b}}{1-\mathbf{b}^2} |df|\leq \kappa_F |df|.
\]

Also, we have $F^*(df)\geq (1-\mathbf{b})|df| \geq \kappa_F |df|$.

\end{proof}

As a direct consequence, we have

\begin{corollary}
 Let $(M,g)$ be a Riemannian manifold of dimension $n$ and $(\beta_k)_k$ be a sequence of $1$-forms, with $\|\beta_k\|<1$ for all $k$, 
 converging to the null $1$-form in $\Lambda^1(M)$. Consider the corres\-ponding sequence of Finsler metrics 
 $(F_k)_k$ with $F_k:=\sqrt{g}+\beta_k$. \\
 Then the real sequence of first eigenvalues $\mu_k = \lambda_{1,p}(M,F_k)$ converges to $\mu = \lambda_{1,p}(M,g)$.
\end{corollary}

\begin{proof}

For all $k$, we have 
\[
 \frac{1-\mathbf{b}_k}{1+\mathbf{b}_k}\leq \frac{\lambda_{1,p}(M,F_k)}{\lambda_{1,p}(M,g)} \leq \frac{1+\mathbf{b}_k}{1-\mathbf{b}_k}
\]
Since $\beta_k \longrightarrow 0$ then $\mathbf{b}_k \longrightarrow 0$. Hence
\[
 \lim_{k\rightarrow \infty} \frac{\lambda_{1,p}(M,F_k)}{\lambda_{1,p}(M,g)} = 1.
\]

\end{proof}

\begin{corollary}
 Let $(M,F:=\sqrt{g}+\beta)$ be a compact Randers manifold. For any $p,q \in \mathbb{R}$ such that $1<p\leq q$, the positive eigenvalues 
 $\lambda_{1,p}(M,F)$ and $\lambda_{1,p}(M,F)$ satisfy 
 \[
  \frac{p\sqrt[p]{\lambda_{1,p}(M,F)}}{q\sqrt[q]{\lambda_{1,q}(M,F)}} \leq \sigma_F.
 \]

\end{corollary}

\begin{proof}

Let $1<p<q$. By Proposition \ref{pp:rand}, we obtain
\[
  \frac{p\sqrt[p]{\lambda_{1,p}(M,F)}}{q\sqrt[q]{\lambda_{1,q}(M,F)}} \leq \kappa_F^2 \frac{p\sqrt[p]{\lambda_{1,p}(M,g)}}{q\sqrt[q]{\lambda_{1,q}(M,g)}}.
\]
However, the map $t \mapsto t\sqrt[t]{\lambda_{1,t}(M,g)}$ is strictly increasing on $(1,\infty)$ (see \cite{Mat05}). Then,
 \[
  \frac{p\sqrt[p]{\lambda_{1,p}(M,F)}}{q\sqrt[q]{\lambda_{1,q}(M,F)}} \leq \kappa_F^2= \sigma_F.
 \]

\end{proof}

\section{Cheeger-type inequality}
\begin{definition}
Let $(M,F,d\mathfrak{m})$ be a closed $n$-dimensional Finsler manifold. The
 Cheeger's constant is defined by
 \begin{equation}
  \mathbf{h}(M):=\inf_{\Gamma}\frac{\min\{ A_{\pm}(\Gamma)\}}{\min\{ \mathfrak{m}(D_1),\mathfrak{m}(D_2)\}},
 \end{equation}
where $\Gamma$ varies over $(n-1)$-dimensional submanifolds of $M$ which divide $M$ into disjoint open submanifolds $D_1$, $D_2$ of $M$ 
with common boundary $\partial D_1=\partial D_2=\Gamma$. One denotes $A_{\pm}(\Gamma)$ the areas of $\Gamma$ induced by the outward and inward normal 
vector field $\mathbf{n}_{\pm}$.
\end{definition}

We have the following usefull co-area formula:

\begin{lemma}\cite{YZ13}\label{area}
 Let $(M,F,\mathfrak{m})$ be a Finsler measure space. Let $\phi$ be a piecewise $C^1$ function on $M$ such that $\phi^{-1}(\{t\})$ is compact for all 
 $t \in \mathbb{R}$.
 Then for any continuous function $f$ on $M$, we have
 \[
  \int_M f F(\nabla\phi)\ d\mathfrak{m} = \int_{-\infty}^{\infty} \left( \int_{\phi^{-1}(t)}f\ dA_\mathbf{n} \right)\ dt,
 \]
where $\mathbf{n}:=\nabla\phi/F(\nabla\phi)$.
\end{lemma}

Lemma \ref{area} yields the following :

\begin{lemma}
 Given a positive function $f\in C^1(M)$. Then, we have
\[
 \int_M F(\nabla f)\ d\mathfrak{m} \geq \mathbf{h}(M) \int_M f\ d\mathfrak{m}.
\]
\end{lemma}

\begin{proof}

Let $f\in C^1(M)$. From Lemma \ref{area}, we have
\begin{eqnarray*}
 \int_M F(\nabla f)\ d\mathfrak{m} &=& \int_0^\infty \left( \int_{f^{-1}(t)} dA_n \right)\ dt\\
                                 &=& \int_0^\infty A_n(\{f=t\})\ dt\\
                                 &=& \int_0^\infty \frac{A_n(\{f=t\}}{\mathfrak{m}(\{f\geq t\})}.\mathfrak{m}(\{f\geq t\})\ dt\\
                                 &\geq& \inf_t \frac{A_n(\{f=t\}}{\mathfrak{m}(\{f\geq t\})} \int_0^\infty \mathfrak{m}(\{f\geq t\})\ dt\\
                                 &\geq& \mathbf{h}(M) \int_M f\ d\mathfrak{m}.
\end{eqnarray*}

\end{proof}

We now state our Cheeger-type inequality: 

\begin{theorem}\label{cheeger1}
 Let $(M,F,\mathfrak{m})$ be a closed Finsler manifold such that $\sigma_F\leq \sigma$. Then
 \[
  \lambda_{1,p}(M)\geq \left( \frac{\mathbf{h}(M)}{\sigma p} \right)^{p}.
 \]

\end{theorem}

\begin{proof}

 Let $f$ be a smooth function on $M$. Define $f_+:= \max\{ f,0\}$ and $f_-:= \max\{ -f,0\}$. Then
 \begin{eqnarray*}
  \mathbf{h}(M)\int_M |f|^p\ d\mathfrak{m} &=& \mathbf{h}(M) \left(\int_M f_+^p\ d\mathfrak{m} + \int_M f_-^p\ d\mathfrak{m}\right)\\
                                          &\leq& \int_M F^*(Df_+^p)\ d\mathfrak{m} + \int_M F^*(Df_-^p)\ d\mathfrak{m}\\
                                          &=& p\left[ \int_M f_+^{p-1} F^*(Df_+)\ d\mathfrak{m}+\int_M f_-^{p-1}F^*(Df_-)\ d\mathfrak{m} \right]\\
                                          &\leq& p\sigma_F \int_M |f|^{p-1} F^*(Df) d\mathfrak{m}\\
                                          &\leq& p\sigma \left( \int_M |f|^p\ d\mathfrak{m}\right)^{\frac{p-1}{p}}
                                          \left( \int_M F^*(Df)^p\ d\mathfrak{m}\right)^{\frac{1}{p}}.
 \end{eqnarray*}
Hence,
\[
 \int_M F^*(Df)^p\ d\mathfrak{m} \geq \left(\frac{\mathbf{h}}{p\sigma}\right)^p \int_M |f|^p\ d\mathfrak{m}. 
\]
Taking the infimum over $\mathcal{H}_0^p(M)$, the inequality follows.

\end{proof}

In \cite{Yau}, Yau showed that on a $n$-dimensional compact Riemannian manifold without boun\-dary whose Ricci curvature is bounded from below by 
$(n-1)K$, the first eigenvalue can be bounded from below in terms of the diameter, the volume of the manifold and the constant $K$. The authors of 
\cite{YZ13} gave a finslerian version of this result for the non-linear Shen's Laplacian. As in \cite{YZ13}, we use the following Croke-type inequality
to obtain the general case:

\begin{proposition}\cite{ZS}\label{pp:yau}
Let $(M,F,d\mathfrak{m})$ be a closed Finsler $n$-dimensional manifold  satisfying $Ric\geq (n-1)K$ for some constant $K$, 
where $d\mathfrak{m}$ denotes either the Busemann-Hausdorff measure or 
the Holmes-Thompson
 measure. Then
 \[
  \mathbf{h}(M)\geq \frac{(n-1)\mathfrak{m}(M)}{2Vol(\mathbb{S}^{n-2})\sigma_F^{4n+\frac{1}{2}}diam(M)\int_0^{diam(M)}\mathfrak{s}_K^{n-1}(t)\ dt},
 \]
where $diam(M)$ denotes the diameter of $M$ and the function $\mathfrak{s}_K$ is defined by 
\[
 \mathfrak{s}_K(t):= \left\{ 
 \begin{array}{l}
  \frac{1}{\sqrt{K}} \sin(\sqrt{K}t),\hspace{0.8cm}  K>0,\\
  t, \hspace{3cm} K=0,\\
  \frac{1}{\sqrt{-K}} \sinh(\sqrt{-K}t),\  K<0.
 \end{array}
\right.
\]

\end{proposition}

From Theorem \ref{cheeger1} and Proposition \ref{pp:yau}, we obtain the following Yau-type estimate.

\begin{proposition}
 Let $(M,F,d\mathfrak{m})$ be a $n$-dimensional closed Finsler manifold whose Ricci curvature satisfies $Ric\geq (n-1)K$ for some real constant $K$, 
 where $d\mathfrak{m}$ denotes either the Busemann-Hausdorff measure or 
the Holmes-Thompson
 measure. Then
 \[
  \lambda_{1,p}(M)\geq \left(\frac{(n-1)\mathfrak{m}(M)}{2pVol(\mathbb{S}^{n-2})\sigma_F^{4n+\frac{3}{2}}diam(M)\int_0^{diam(M)}\mathfrak{s}_K^{n-1}(t)\ dt}
  \right)^p.
 \]
\end{proposition}

\begin{proof}

 By the propositionosition \ref{pp:yau}, we have
 \[
  \frac{\mathbf{h}(M)}{p \sigma_F}\geq \frac{(n-1)\mathfrak{m}(M)}{2pVol(\mathbb{S}^{n-2})
  \sigma_F^{4n+\frac{3}{2}}diam(M)\int_0^{diam(M)}\mathfrak{s}_K^{n-1}(t)\ dt}.
 \]
A direct application of the theorem \ref{cheeger1} completes the proof.
  
\end{proof}


\vspace{1cm}
{\bf Cyrille Combete}

Institut de Mathematiques et de Sciences Physiques, 
 B.P 613, Porto-Novo, Benin\\
{ \bf Email:} cyrille.combete@imsp-uac.org

\vspace{1cm}

{\bf Serge Degla}

Institut de Mathematiques et de Sciences Physiques,
 B.P 613, Porto-Novo, Benin\\
\& Ecole Normale Superieure, 
 B.P 72, Natitingou, Benin\\
{\bf Email:} sdegla@imsp-uac.org

\vspace{1cm}

{\bf Leonard Todjihounde}

Institut de Mathematiques et de Sciences Physiques, 
 B.P 613, Porto-Novo, Benin\\
{\bf Email:} leonardt@imsp-uac.org


\begin{thebibliography}{0}
\bibitem{AZ} Antonelli, P.L., Zastawniak,  T.J: \emph{Diffusion, Laplacian and Hodge decomposition on Finsler Spaces}, 
The theory of Finsler Laplacians and Applications, Kluwer Acad. Publ.,  {\bf 459}, p 141-149, (1998)
\bibitem{BL} Bao, D., Lackey, B.: \emph{A Hodge decomposition theorem for Finsler spaces}, C.R. Acad. Sc., Paris, {\bf 323}, p 51-56, (1996).
\bibitem{Bar} Barthelm\'e, T.: \emph{A new Laplace operator in Finsler geometry and periodic orbits of Anosov flows}, PhD Thesis, University of Lausane, (2012).
\bibitem{Cen}  Centor\'e, P.:  \emph{A mean-value Laplacian for Finsler spaces}, Doctoral Thesis, University of Toronto, 1998. Lectures in Mathematics,Birkhauser, (2005).
\bibitem{Ch} Cheeger,  J.:  \emph{A lower bound for the smallest eigenvalue of the Laplacian}, Problems in Analysis, 
Symposium in honor of  S. Bochner, Princeton univ. Press, Princeton, NJ, pp. 195-199, (1970).

\bibitem{HY14}He, Q., Yin, S-T.: \emph{The first eigenvalue of Finsler p-Laplacian.} Diff. Geo. Appli {\bf 35}, 30-49,(2014)

\bibitem{HY16}He, Q., Yin, S-T.: \emph{The first eigenfunctions and eigenvalue of the p-Laplacian on Finsler manifolds.}
Sci. China Math, {\bf59 }, (2012).
\bibitem{Mat05} Matei, A-M: \emph{Boundedness of the first eigenvalue of $p$-Laplacian}, AMS,{\bf 133}, Num. 7, p 2183-2192, (2005).
\bibitem{Mat12} Matei, A-M.: \emph{Conformal bound for the first eigenvalue of the $p$-laplacian.} Nonlinear Anal., {\bf 80}, 88-95, (2013). 
\bibitem{MT1} Matveev, V., Troyanov, M.: \emph{The Binet-Legendre metric in Finsler geometry.}, Geom.Topol. 
{\bf 16}, no. 4, 2135-2170, (2012).
\bibitem{MT2}Matveev, V., Troyanov, M.: \emph{Completeness and incompleteness of the Binet-Legendre metric.} Euro. J. Math., {\bf  1}, $483-502$, (2015).
\bibitem{Mun} Munteanu O., \emph{eigenvalue estimates for the Laplacian of Finsler spaces of Randers type}, Houston J. Math, {\bf 37}, no. 2 (2011).
\bibitem{Oht16} Ohta, S. : \emph{A semigroup approch to Finsler geometry: Bakry-Ledoux isoperimetric inequality} preprint (2016) 
\url{arXiv:1602.00390}
\bibitem{Sh98} Shen, Z.: \emph{The non-linear Laplacian for Finsler manifold} The theory of Finsler Laplacians and Applications, Kluwer Acad. Publ., 
{\bf 459}, (1998).

\bibitem{Shen}Shen, Z.: \emph{Lectures on Finsler geometry}, World Scientific Publishing, (2001).
\bibitem{ZS} Shen, Y., Zhao, W.: \emph{A universal volume comparison theorem For Finsler manifold and related result}, Can. J. Math.,{\bf 65},1401-1435, (2013)
\bibitem{Yau} Yau, S-T.: \emph{Isoperimetric constants and the first eigenvalue of a compact Riemannian manifold}, Anna. Scien. de l'ENS, {\bf8}, no. 4,
page 487-507, (1975).
\bibitem{YZ13} Yuan, L.,  Zhao, W.: \emph{Cheeger's constant and the first eigenvalue of a closed Finsler manifold}, preprint ( 2013) 
\url{arXiv1309.2115v2}.
\end{thebibliography}
\end{document}